\documentclass[11pt]{amsart}

\usepackage[applemac]{inputenc}
\usepackage{color}
\usepackage{bibtopic}
\usepackage{amsmath,amssymb,amsthm,amscd}

\newtheorem{thm}{Theorem} [section]
\newtheorem{cor}[thm]{Corollary}
\newtheorem{lem}[thm]{Lemma}

\theoremstyle{definition}

\theoremstyle{theorem}
\newtheorem{rem}[thm]{Remark}

\renewcommand{\>}{\rangle}
\newcommand{\<}{\langle}

\renewcommand{\epsilon}{\varepsilon}

\newcommand{\Con}{{\rm Con}^\omega }
\newcommand{\CG}{\Con (G, d)}
\newcommand{\e }{\varepsilon }
\newcommand{\lab }{{\bf Lab}}
\date{}
\begin{document}

\title{Finitely presented groups with infinitely many non-homeomorphic asymptotic cones}

\author{D. Osin, A. Ould Houcine}
\address[D. Osin]{Department of Mathematics, Vanderbilt University, Nashville TN 37240, USA.}
\email{denis.osin@gmail.com}
\address[A. Ould Houcine]{Universit\'e de Mons, Institut de Math\'ematique, B\^atiment Le Pentagone, avenue du Champ de Mars 6, B-7000 Mons, Belgique.
Universit\'e de Lyon; Universit\'e Lyon 1; INSA de Lyon, F-69621; Ecole Centrale de
Lyon; CNRS, UMR5208, Institut Camille Jordan, 43 Blvd du 11 Novembre 1918,
F-69622 Villeurbanne-Cedex, France.}
\email{ould@math.univ-lyon1.fr}

\thanks{The research was supported by the NSF grant DMS-1006345. The first author was also supported by the RFBR grant 11-01-00945}

\maketitle

\begin{abstract} We construct a finitely presented group with infinitely many non-homeomorphic asymptotic cones.   We also show that the existence of cut points in asymptotic cones of finitely presented groups does, in general, depend on the choice of scaling constants and ultrafilters.
\end{abstract}

\section{Introduction}

Asymptotic cones where first introduced by Gromov \cite{Gro-asy} to prove virtual nilpotence of groups of polynomial growth.  Van den Dries and Wilkie \cite{Van-Wilkie} gave a definition, which applies to arbitrary metric spaces and uses non-standard analysis, via ultrafilters.  Roughly speaking, the asymptotic cone of a metric space $(S,dist)$ corresponding to a non-principal ultrafilter $\omega$, a sequence of observation points $e=(e_n)_{n \in \Bbb N}\subseteq S$, and a sequence of scaling constants $d=(d_n)_{n \in \Bbb N}$ diverging to $\infty$, is the ultralimit of the pointed metric spaces $(S, dist/d_n, e_n)$.

If  $G$ is a group endowed with a word metric with respect to a finite generating set, then asymptotic cones of $G$ are independent on the choice of the observation sequence $e$. Moreover, up to bi-Lipschitz equivalence asymptotic cones of $G$ are independent of the choice of a particular finite generating set of $G$. In what follows, we denote the asymptotic cone of a finitely generated group $G$ corresponding to a non-principal ultrafilter $\omega$ and a sequence of scaling constants $d$ by $\CG$.

Many algebraic, geometric, and algorithmic properties of finitely generated groups are encoded in topology of their asymptotic cones. For instance, a group $G$ is hyperbolic if and only if all asymptotic cones of $G$ are real trees \cite{Gro}. All asymptotic cones of $G$ are proper if and only if $G$ is virtually nilpotent \cite{Van-Wilkie,Dru}. If all asymptotic cones of $G$ are simply connected, then $G$ is finitely presented and satisfies a polynomial isoperimetric inequality; in particular, the word problem in $G$ is in NP \cite{Dru,Gro-que}.

It is natural to ask whether various topological invariants of asymptotic cones of a given group $G$ depend on the choice of the scaling sequence and the ultrafilter. Questions of this kind go back to \cite{Gro-que},  where Gromow asked if a finitely generated (or finitely presented) group can have two non-homeomorphic asymptotic cones. First examples of finitely generated groups with this property where constructed by Thomas and Velicovic in \cite{Thom-Veli}. Later on Drutu  and  Sapir \cite{DS}  gave an example of a finitely generated group with continuously many non-homeomorphic asymptotic cones. In both constructions the desired groups are limits of small cancellation groups. In particular, groups constructed by Thomas-Velicovich and Drutu-Sapir are not finitely presented.

Some progress on Gromov's question for finitely presented groups was achieved by Kramer, Shelah, Tent and Thomas  \cite{K-S-T-T}. They proved that some natural finitely presented groups (e.g., uniform lattices in $SL_3(\mathbb R)$) have a unique asymptotic cone up to homeomorphism if the Continuum Hypothesis holds, and $2^{2^{\aleph_0}}$ non-homeomorphic asymptotic cones if the Continuum Hypothesis fails.  The first and the only example of a finitely presented group with different asymptotic cones independent of the continuum Hypothesis was constructed by Olshanskii and Sapir \cite{Ol-sap}. More precisely, they constructed a finitely presented group having two asymptotic cones, one of which is simply connected while the other is not. However the question of whether a finitely presented group can have infinitely many non-homeomorphic asymptotic cones independently of the Continuum Hypothesis was open until now.

To distinguish between infinitely many asymptotic cones one would need a topological invariant which can take infinitely many values. Most traditional algebraic invariants (e.g., the fundamental group, which is used in \cite{DS}) are, in general, very hard to compute for asymptotic cones of groups. In this paper we use \emph{ connectivity degree} of a path connected metric space $S$, denoted $c(S)$, which is the minimal cardinality of a finite subset $D\subset S$ such that $S\setminus D$ is path disconnected. If $S$ can not be disconnected by removing a finite subset, we set $c(S)=\infty $.   Our main results is the following. We denote by $\mathcal P$ the set of all prime numbers.

\begin{thm}\label{main} There exists a finitely presented group $\Gamma $  satisfying the following condition. For every $p\in \mathcal P\cup \{ 1\}$ there exists a scaling sequence $d=(d_n)$ such that for every non-principal ultrafilter $\omega $, we have $c(\Con (\Gamma, d))=p$. \end{thm}

Since connectivity degree is invariant under homeomorphisms, the next result is an immediate corollary of the theorem.

\begin{cor} There exists a finitely presented group with infinitely many non-homeomorphic asymptotic cones. \end{cor}

The same construction also allows us to address the question of whether the existence of cut points in asymptotic cones of finitely presented groups depends on the choice of scaling constants and ultrafilters.  Recall that $s$ is a cut point of a path connected metric space $S$ if $S\setminus{s}$ is path disconnected. Thus $S$ has a cut point iff $c(S)=1$. Existence of cut points in asymptotic cones of a given group $G$ has many purely algebraic consequences (e.g., such a group $G$ does not satisfy any nontrivial law) and can be used to study outer automorphisms and subgroups of $G$   \cite{DS,DMS}. Recall that a finitely generated group is \textit{wide} if all its asymptotic cones are without cut points and is {\it unconstricted} if at least one of its asymptotic cones does not have cut points. Drutu and Sapir \cite{DS} asked if every unconstricted group is wide. In \cite{Ol-Os-S}, Olshanskii,  Osin  and  Sapir constructed first examples of finitely generated unconstricted non-wide  groups. However these groups are not finitely presented.

For finitely presented groups the question is of particular interest since the property of being wide is closely related to the existence of the so-called Morse quasi-geodesics \cite{DMS}, which can be thought of as hyperbolicity of the group along a certain direction. Notice that the ordinary hyperbolicity of a finitely presented group can be recognized by looking at just one asymptotic cone. Indeed in the appendix to \cite{Ol-Os-S}, Kapovich and Kleiner proved that if at least one asymptotic cone of a finitely presented group $G$ is a real tree, then $G$ is hyperbolic. On the other hand this is not true for finitely generated groups \cite{Ol-Os-S}. Thus, a priori, the answers to the Drutu-Sapir question could be different for finitely generated and finitely presented groups. However, Theorem \ref{main} obviously implies the following.

\begin{cor} There exists a finitely presented group $\Gamma $ and sequences $a=(a_n)$, $b=(b_n)$ such that for any non-principal ultrafilter $\omega$, \emph{Con}$^\omega(\Gamma , a)$ has cut points while \emph{Con}$^\omega(\Gamma , b)$ does not. \end{cor}

There are two main ingredients in the proof of Theorem \ref{main}. The first one is a refined version of a construction from \cite{Ol-Os-S}, which makes use of central extensions to produce an (infinitely presented) group with  asymptotic cones of different connectivity degree. Then we use an improved center-preserving version of the Higman embedding which pursues work of  the second author in \cite{Ould} to obtain the desired finitely presented group. Our key addition to the results from \cite{Ould} is part (c) of Theorem \ref{thm-embedding}, which provides a uniform (in fact, quadratic) estimate of the distortion of the embedded subgroup and allows us to control the asymptotic geometry of the resulting finitely presented group.

The paper is organized as follows. In the next section we define asymptotic cones and collect some results about cut points and connectivity degree used in the proof of Theorem \ref{main}. In Section 3 we discuss the center-preserving version of the Higman embedding with quadratic distortion. The main construction and the proof of Theorem \ref{main} is contained in Section 4.

\section{Cut points in asymptotic cones and HNN-extensions}

Given a word $W$ in some alphabet, we denote by $\| W\| $ its length. If $X$ is a generating set of a group $G$, we do not distinguish between words in $X\cup X^{-1}$ and elements of $G$ represented by these words if no confusion is possible. We write $W\equiv V$ to express the letter--for--letter equality of words $W$ and $V$ in $X\cup X^{-1}$ and $W=V$ if $W$ and $V$ represent the same element of the group $G$. We also denote by $|g|_X$ the (word) length  of an element $g\in G$ and by $\Gamma (G,X)$ the Cayley graph of $G$ with respect to $X$. The formula $dist_X(h,g)=|h^{-1}g|_X$ defines a metric on $G$, called the \textit{word metric} (relative to $X$). Given a path $p$ in $\Gamma (G,X)$, $\ell (p) $ denotes its length and $p_-,p_+$ denote the beginning and the ending points of $p$, respectively.

A \textit{non-principal ultrafilter} $\omega$ on $\Bbb N$ is a finitely additive measure defined on all subsets $S$ of $\Bbb N$, such that $\omega(S) \in \{0, 1\}$, $\omega(\Bbb N)=1$, and $\omega(S)=1$ whenever $S$ is finite.  Given a bounded sequence of real numbers $(a_n)$, there exists a unique real number $a$ satisfying $\omega(\{n \in \Bbb N :  |a_n-a|<\epsilon\})=1$ for every $\epsilon >0$, called the \emph{limit of $(a_n)$ with respect to $\omega$} and denoted by $\lim^\omega a_n$.  Similarly, $\lim^\omega a_n=\infty$ if $\omega(\{n \in \Bbb N : a_n >M\})=1$ for every $M>0$.

Given two infinite sequences of real numbers $(a_n)$ and $(b_n)$ we write $a_n=o_\omega(b_n)$ if $\lim^\omega (a_n/b_n)=0$. Similarly, $a_n=\Theta_\omega(b_n)$ (respectively $a_n=O_\omega(b_n)$) means that $0<\lim^\omega (a_n/b_n)<\infty$ (respectively $\lim^\omega (a_n/b_n)<\infty$).

Let $(X_n, dist_n)_{n \in \Bbb N}$ be a sequence of metric spaces.  Fix an arbitrary sequence $e=(e_n)$ of points $e_n \in X_n$ called {\it observation points}. Consider the set $\mathcal F_e$ of sequences $g=(g_n)$, $g_n \in X_n$, such that $dist_n(g_n,e_n)\leq c$ for some constant $c=c(g)$.  Two sequences $(f_n)$ and $(g_n)$ of $\mathcal F_e$ are said to be \textit{equivalent} if $\lim^\omega dist_n(f_n, g_n)=0$. The equivalence class of $(g_n)$ is denoted by $(g_n)^\omega$. The \textit{$\omega$-limit of pointed metric spaces $(X_n, dist_n, e_n)$}, denoted by $\lim^\omega(X_n)_e$, is the quotient space of the equivalence classes where the distance between $f=(f_n)^\omega$ and $g=(g_n)^\omega$ is defined by $$dist (f,g)={\lim }^\omega dist_n(f_n,g_n).$$

An \textit{asymptotic cone}  Con$^\omega(X,e,d)$ of a metric space $(X, dist)$, where $e=(e_n)\subseteq X$ is a sequence of observation points and $d=(d_n)$ is an unbounded non-decreasing \textit{scaling sequence} of positive real numbers, is the $\omega$-limit of pointed spaces $X_n=(X, dist/d_n, e_n)$. The asymptotic cone is a complete space;  it is a geodesic metric space whenever $X$ is. We note that Con$^\omega(X, e,d)$ does not depend on the choice of $e$ if $X$ is homogeneous. For instance this the case if $X$ is a finitely generated group with a word metric, so in this case we will omit $e$ from the notation.  Clearly the asymptotic cone of a finitely generated group endowed with a word metric coincides with the corresponding asymptotic cone of its Cayley graph.

We start with a sufficient condition for the existence of cut points in asymptotic cones. The lemma below is an immediate corollary of the equivalence of conditions (2) and (3) from Proposition 3.24 in \cite{DMS}. Given a path $q$ in a metric space and two points $x,y\in q$, we denote by $q_{xy}$ the maximal subpath of $q$ connecting $x$ to $y$. Recall that a finitely generated group is {\it constricted} if all   its asymptotic cones have cut points.

\begin{lem}[Drutu, Mozes, Sapir {\cite{DMS}}]\label{DMS}
Let $G$ be a group generated by a finite set $X$. Suppose that the Cayley graph $\Gamma (G, X)$ contains a bi-infinite quasi-geodesic $q$ satisfying the following property. For every $C\ge 1$,  there exists $D\ge 0$ such that for every two vertices $x,y\in q$, every path of length at most $Cdist(x,y)$ connecting $x$ and  $y$ crosses the $D$-neighborhood of the middle third of $q_{xy}$. Then $G$ is constricted.
\end{lem}

The main result of this section is the following.

\begin{lem}\label{cut points in HNN}
Let $B$ be a group generated by a finite set $Y$, $A\le B$. Let $$U=\langle B, t\mid a^t=a, \; a\in A\rangle . $$ We endow $B$ and $U$ with word metrics with respect to $Y$ and $Y\cup \{ t\}$, respectively.
Then the following hold.
\begin{enumerate}
\item[(a)] For every $b\in B$ we have $|b|_Y=|b|_{Y\cup\{t\}}$. In particular, the inclusion $B\le U$ induces an isometric embedding $\Con (B,d)\to  \Con (U,d)$ for every $d$ and $\omega$.

\item[(b)] Suppose that there exists $b\in B$ such that $A^b\cap A=\{ 1\}$.
Then $U$ is constricted.
\end{enumerate}
\end{lem}

\begin{proof}
To prove part (a), observe that  if an element $b\in U$ is represented as a word $W$ in the alphabet $Y^{\pm 1}\cup \{ t^{\pm 1}\}$, then $t$-reductions (i.e., passing from subwords of the form $t^{-1}at$ or $tat^{-1}$ to $a$) decrease the length of $W$. Hence the shortest word $W_0$ representing $b$ is necessarily reduced, i.e., does not contains subwords of the form $t^{-1}at$ or $tat^{-1}$ . If $b\in B$, then the Britton lemma easily implies that $W_0$ does not contain $t^{\pm 1}$, i.e., $W_0$ is a word in $Y^{\pm 1}$. Hence $|b|_Y=|b|_{Y\cup\{t\}}$. The statement ``in particular" follows immediately from the definition of an asymptotic cone.

To prove (b) we first note that if $A=\{ 1\} $, then the lemma is trivial as $U$ is a nontrivial free product and hence it is hyperbolic relative to a proper subgroup in this case. The later condition implies the existence of cut points \cite{DS}. In what follows we assume that $A\ne \{ 1\} $ and, in particular, $b\notin A$.

Without loss of generality we may assume that $b\in Y$. Let $X=Y\cup \{ t\}$. By Lemma~\ref{DMS} it suffices to show that for every $C$, there exists $D\ge 0$ such that every path $r$ in $\Gamma (U, X)$ labeled by a power of $(tb)^3$ is geodesic and any other path $p$ connecting $r_-$ to $r_+$ of length
\begin{equation}\label{lple}
\ell(p)\le C\ell (r)
\end{equation}
intersects the $D$-neighborhood of the middle third of $r$. Indeed then any bi-infinite path $q$ labelled by $\cdots tbtb\cdots $ satisfies the assumptions of Lemma \ref{DMS}.

Let $r$ and $p$ be as above. Let $\lab (r)\equiv (tb)^{3n}$ and let $\lab (p)\equiv W$. Thus $(tb)^{-3n}W=1$ in $G$. Since $b\notin A$, the Britton Lemma \cite{shu-lyn} implies that $W\equiv w_0 t w_1 t\ldots w_{3n-1}t w_{3n}$ for some words $w_0, \ldots , w_{3n}$ in the alphabet $Y\cup Y^{-1}$ such that for every $0\le k\le 3n$, the word
\begin{equation}\label{ak}
a_k\equiv (tb)^{-k} w_0t\ldots w_{k-1}tw_k
\end{equation}
represents an element of $A$. Note that $w_0, \ldots , w_{3n}$ may contain $t^{\pm 1}$. In the group $U$ we have
\begin{equation}\label{wk}
\begin{array}{rcl}
  a_{k-1}^{-1} ba_k & = & t^{-1}a_{k-1}^{-1} t ba_k =\\&&\\
  &  & t^{-1}\large(w_{k-1}^{-1}t^{-1}\ldots w_1^{-1}t^{-1}w_0^{-1} (tb)^{k-1}\large) t b \large((tb)^{-k} w_0t\ldots w_{k-1}tw_k\large) =w_k.
\end{array}
\end{equation}
In particular, $w_k\ne 1$ in $U$ for all $0<k\le 3n$ as $b\notin A$. Hence $\ell (p)=\| W\| \ge 6n=\ell (r)$, i.e., $r$ is geodesic.

Observe that there exists $D>0$ such that for every $a, a^\prime \in A$, the inequality $\max \{ |a|_Y, |a^\prime |_Y\} \ge D$ implies $|aba^\prime|_Y\ge 6C$. Indeed otherwise there would exist distinct pairs $(a_1, a_1^\prime ), (a_2, a_2^\prime )\in A\times A$ such that $a_1ba^{\prime}_1=a_2ba_2^{\prime}$. This would imply $(a_2^{-1}a_1)^b=a_2^\prime (a_1^\prime )^{-1}$, which contradicts the assumption $A^b\cap A=\{ 1\} $.

Let $D$ be chosen to satisfy the above condition. If for every $n\le k\le 2n$, we have $|a_k|_Y\ge D$, then $|w_k|_Y\ge 6C$ for every $n+1\le k\le 2n$ by (\ref{wk}) and the choice of $D$. Therefore, $$\ell (p)=\| W \| > \sum\limits_{k=n+1}^{2n} |w_k|_Y \ge 6Cn =C\ell (r),$$ which contradicts (\ref{lple}). Hence $|a_k|_Y< D$ for some $n\le k\le 2n$. By (\ref{ak}) this means that $p$ intersects the $D$-neighborhood of the middle third of $r$.
\end{proof}

\section{Quadratically distorted center-preserving Higman embeddings}

The famous Higman theorem states that every recursively presented group embeds in a finitely presented one. This result was significantly improved by Olshanskii \cite{Ol}, who proved the following.

\begin{thm} [Olshanskii{ \cite[Theorem 3]{Ol}}]  \label{thm-Ol} Let $G$ be a group with a finite generating set $X$ and a recursively enumerable set of defining relations. Then there exists an isomorphic embedding of $G$ in a finitely presented group $H$  generated by a finite set $Y$ such that $|g|_X=|g|_Y$ for each $g \in G$.
\end{thm}

Another improvement of the Higman's theorem was obtained in \cite{Ould}, where the second author proved that every finitely generated recursively presented group $G$ embeds into a finitely presented group $H$ in such a way that the center of $G$ coincides with that of $H$. The main result of this section combines the main features of both improvements, although our distortion estimate is not as good as in the Olshanskii's theorem.

\begin{thm}  \label{thm-embedding} Let $G$ be finitely generated recursively presented group with a finite generating set $A$. Then $G$ embeds into  a group $\Gamma$ with a finite generating set $B$ such that the following conditions hold.
\begin{enumerate}
\item[(a)] $\Gamma $ is finitely presented.

\item[(b)] $Z(G)=Z(\Gamma)$.

\item[(c)] For any $g \in G$, $\sqrt{|g|_A }\leq |g|_B \leq |g|_A$.

\item[(d)] $\Gamma /Z(\Gamma )$ is constricted.
\end{enumerate}
\end{thm}

\begin{proof}
The proof of the theorem is based on the construction from \cite{Ould}; we recall  it briefly in what follows. For more details and proofs the reader is referred to \cite{Ould}. Multiplying $G$ directly by a finitely generated recursively presented centerless non-abelian group (say, $S_3$), we can assume that $G$ is non-abelian. In particular, $Z(G)\lvertneqq G$. Let
$$
G_0=\<G, z| g^z=g, g \in Z(G)\>.
$$

Note that an element of a group belongs to its center if and only if it commutes with all generators. Since $G$ is finitely generated and recursively presented, there is an obvious algorithm which enumerates all elements of  $Z(G)$. Hence  the group $G_0$ is recursively presented. Since  $Z(G)\lvertneqq G$, we obviously have $Z(G)=Z(G_0)$.  Let $A=\{a_1, \dots, a_n\}$ and let $F_X$ be the free group with basis $X=\{x_1, \dots, x_n, x_{n+1}\}$. Thus there is an isomorphism  $v : F_X /R \rightarrow G_0$, where $v (x_i ) = a_i$ for $i = 1, \dots ,n$, $v (x_{n+1}) = z$, and
$R$ is the normal closure of the the set of relations of $G_0$ (which is recursively enumerable).   Let
$$
F_R=\<F_X, d| r^d=r, r\in R\>.
$$

If $w$ is a word in the generators of $F_X$, let $\bar w$ denote the word of $G_0$ obtained by replacing   $x_i$ with $a_i$ for $i=1, \dots, n$, and $x_{n+1}$ with $z$. Let $L=\langle F_X\cup F_X^d\rangle \le F_R$. Obviously $L$ is the free product of $F_X$ and $F_X^d$ with $R$ amalgamated.  Thus the map $\phi $ defined by $\phi (w)=\bar w$ and $\phi(w^d)=1$ for every $w\in F_X$ extends to a homomorphism $\phi : L \rightarrow G_0 $.

Define a map $\psi : L \times Z(G_0) \rightarrow L \times G_0$ by
$$
\psi(l,g)=(l, \phi(l)g).
$$
Then $\psi$ is an injective homomorphism. By Higman's embedding theorem, we can embed $F_R$ into a finitely presented group $H$.

\begin{rem}\label{rem1}
The only property of the group $H$ which is used in \cite{Ould} is finite presentability; the particular structure of the group $H$ is completely irrelevant. Hence, by Theorem  \ref{thm-Ol} we can choose $H$ so that it has a finite generating set $T$ such that  $|x|_{X\cup\{d\}}=|x|_T$ for any $x \in F_R$.
\end{rem}

Let
$$
K=\<H \times G_0, s| s^{-1}(l,g)s=(l, \phi(l)g), l \in L, g \in Z(G_0)\>
$$
and let $M=G^\prime \times \<t\> $, where $G^\prime $ is an isomorphic copy of $G$. Viewing $G$ as a subgroup of $G_0$ and hence as a subgroup of $K$, we form the free product with amalgamation
$$
U=K*_G M.
$$
Let
$$
r=(s^{-1}z)\cdot t \cdots z\cdot t^2 \cdots z\cdot t^{80}.
$$

Then the symmetrized set generated by $r$ satisfies the small cancellation condition $C'(1/70)$ with respect to the amalgamated free product structure of $U$ (for details about small cancellation theory  in this context  see \cite{shu-lyn}). Let $N$ be the normal closure of $r$ in $U$.  Then $\Gamma_0=U/N$ is a finitely presented group, $K$ and $M$ embed in $\Gamma _0$, and
\begin{equation}\label{center}
Z(\Gamma_0)=Z(G_0)=Z(G).
\end{equation}
For the detailed proofs of all the above facts, we refer the reader to \cite{Ould}. Finally let $$\Gamma =\< \Gamma _0, q\mid g^q=g,\, g\in G\> .$$  We are going to show that the group $\Gamma $ satisfies conditions (a)--(d). For the convenience of the reader we provide the diagram showing relations between the groups constructed above;  arrows correspond to (isomorphic) embeddings.
$$
\begin{CD}
L @>>> F_R @>>> H @>>> K @>>> \Gamma _0 @>>> \Gamma \\
@. @. @. @AAA @AAA @.\\
@. @. @. G_0 @. M @.\\
@. @. @. @AAA @. @.\\
@. @. @. G @.  @.
\end{CD}
$$

Since $G$ is finitely generated, $\Gamma $ is finitely  presented. Further we note that the quotient group $\Gamma / Z(\Gamma)$ is isomorphic to the HNN-extension of a centerless group with proper associated subgroups. It easily follows from the normal form theorem for HNN-extensions \cite[Theorem 2.1, Ch. IV]{shu-lyn} that $\Gamma / Z(\Gamma)$ is centerless as well. Hence property (b) follows from (\ref{center}).

The proof of (c) is divided into few steps. Let $T$ be a generating set of $H$ chosen according to Remark \ref{rem1}. Set $C=A \cup \{z\}$ and  $D=T \cup A \cup \{z,s\}$. Clearly $C$ and $D$ generate $G_0$ and $K$, respectively.

Below we will often use the following obvious observation without any reference. Let $G_1$ and $G_2$ be groups generated by sets $X_1$ and $X_2$. Endow $G_1$, $G_2$, and $G_1\times G_2$ by the word metrics corresponding to $X_1$, $X_2$, and $X_1\cup X_2$, respectively. Then the natural embeddings $G_i\to G_1\times G_2$, $i=1,2$, are isometric.

\begin{lem}\label{lem-pre-K}
Let $g \in G_0$ and  let $(h_0, g_0), (h_1, g_1), \dots, (h_n, g_n), (h_{n+1}, g_{n+1})$  be a sequence of elements from $H \times G_0$,  such that
$$
g=(h_0, g_0)s^{\epsilon_0} (h_1, g_1)s^{\epsilon_1}\cdots(h_n, g_n)s^{\epsilon_n} (h_{n+1}, g_{n+1})
$$
for some $\epsilon_1, \dots, \epsilon_n \in\{\pm 1\}$.
Then
\begin{equation}\label{gC}
|g|_C \leq n\sum_{0 \leq i \leq n+1}|h_i|_T+\sum_{0 \leq i \leq n+1}|g_i|_C.
\end{equation}
\end{lem}
\begin{proof}
The proof is by induction on $n$. The base of the induction corresponds to $n=-1$ (i.e., $g=(h_0, g_0)$), in which case the statement is obvious.
Suppose now that $n\ge 0$. Then by the Britton Lemma $n\ge 1$ and there exists $j$ such that $(h_j, g_j) \in L \times Z(G_0)$ or $(h_j, g_j) \in \psi(L \times Z(G_0))$ and $s^{\epsilon_{j-1}}(h_j,g_j)s^{\epsilon_j}=(h_j, \phi(h_j)^{\epsilon_j}g_j)$. Thus  we get the new sequence
$$(h_0, g_0), (h_1, g_1), \dots, (h_{j-1}h_jh_{j+1}, g_{j-1}\phi(h_j)^{\epsilon_j}g_jg_{j+1}),\dots,  (h_{n+1}, g_{n+1})$$
of smaller length.
By induction and the triangle inequality, we obtain
$$
|g|_C \leq (n-2)\sum_{0 \leq i \leq n+1}|h_i|_{T}+\sum_{0 \leq i \leq n+1}|g_i|_C+ |\phi(h_j)^{\epsilon_j}|_C.
$$
Note that  $|\phi(h_j)^{\epsilon_j}|_C \leq |h_j|_{X\cup\{ d\}} =|h_j|_{T}$ by the definition of $\phi $. This and the previous inequality imply (\ref{gC}).
\end{proof}

\begin{cor} \label{lem-K}
For any $g \in G_0$,  $|g|_C \leq |g|_D^2$.
\end{cor}

\proof
Let $g \in G_0$. A shortest word in the alphabet $D\cup D^{-1}$ representing $g$ yields a sequence $(h_0, g_0), (h_1, g_1), \dots, (h_n, g_n), (h_{n+1}, g_{n+1})$   as in Lemma \ref{lem-pre-K} such that
\begin{equation}\label{gD}
|g|_D=(n+1)+\sum_{0 \leq i \leq n+1}(|h_i|_{T}+|g_i|_C).
\end{equation}
Combining Lemma \ref{lem-pre-K} and (\ref{gD}), we obtain
$$
|g|_C \leq n\sum_{0 \leq i \leq n+1}|h_i|_T+\sum_{0 \leq i \leq n+1}|g_i|_C\le |g|_D\sum_{0 \leq i \leq n+1}|h_i|_T+\sum_{0 \leq i \leq n+1}|g_i|_C
\leq |g|_D^2
$$
as required.  \qed

Let $B_1= D \cup \{t\}$ and let $\pi : U \rightarrow \Gamma$ be the natural homomorphism. Set $B_2=\pi(B_1)$. Obviously $B_2$ generates  $\Gamma$.

\begin{lem}
\label{lem-U}For any $g \in G$, $|g|_{D}=|\pi(g)|_{B_2}$.
\end{lem}

\proof  Clearly $|\pi(g)|_{B_2} \leq |g|_{D}$ and it remains to show that $|g|_{D} \leq |\pi(g)|_{B_2}$.  Let  $g \in G$ and  $W$ a word  in the alphabet $B_1\cup B_1^{-1}$ such that $\pi(W)=\pi(g)$ and $$|\pi(g)|_{B_2}=\|W\|.$$ Let
$$
W\equiv v_0 t^{\epsilon_0} \cdot v_1t^{\epsilon_1} \cdots v_n t^{\epsilon_n}v_{n+1},
$$
where each $v_i$ is a (possibly empty) word in $D$ and $\epsilon_i =\pm 1$. Obviously the word $W$ is reduced with respect to the (obvious) HNN-structure of $U$ with stable letter $t$. Indeed as the stable letter commutes with the associated subgroup, making reductions decreases the length of the word.

If $g^{-1}w \neq_U1$, then, depending on whether $ g^{-1}v_0 \in G$, $v_{n+1} \in G$ or not, one of the following sequences
$$
(g^{-1}v_0t^{\epsilon_0}, v_1, t^{\epsilon_1}, v_2, \dots, t^{\epsilon_{n-1}}, v_n, t^{\epsilon_n}v_{n+1}),
$$
$$
(g^{-1}v_0, t^{\epsilon_0}, v_1,t^{\epsilon_1}, v_2, \dots, t^{\epsilon_{n-1}}, v_n, t^{\epsilon_n}v_{n+1}),
$$
$$
(g^{-1}v_0, t^{\epsilon_0}, v_1,t^{\epsilon_1}, v_2, \dots, t^{\epsilon_{n-1}}, v_n, t^{\epsilon_n}, v_{n+1}),
$$
$$
(g^{-1}v_0t^{\epsilon_0}, v_1,t^{\epsilon_1}, v_2, \dots, t^{\epsilon_{n-1}}, v_n, t^{\epsilon_n}, v_{n+1})
$$
is reduced with respect to the amalgamated free product structure of $K*_G(G\times \<t|\>)$. By \cite[Theorem 11.2, Chapter V]{shu-lyn}, one of the previous sequences has  a subsequence  of a cyclic permutation   of the sequence
\begin{equation}\label{seq}
(s^{-1}z ,  t ,  z,  t^2 ,  z, \dots,z,  t^{80}),
\end{equation}
whose length is bigger than $(1-3/70)160$.  Since $(1-3/70)160>150>(|r|/2+4)$, we conclude that the sequence
$$
(v_1,t^{\epsilon_1}, v_2, \dots, t^{\epsilon_{n-1}}, v_n)
$$
has  a subsequence of the sequence appearing in (\ref{seq}) whose length is bigger than $|r|/2$.  Since $z \in D$,  replacing that subsequence by the corresponding shorter subsequence, we get a contradiction to the choice of $W$.

Thus $g=_UW$.  Since $W$ is reduced in $U$ with respect to the HNN-structure of $U$, $W$ it does not involve $t$.  Therefore
$$
|\pi(g)|_{B_2}=\|W\| \geq |W|_D=|g|_D,
$$
and we get the required result. \qed

We are now ready to complete the proof of part (c) of Theorem \ref{thm-embedding}. Let $B=B_2\cup \{ q\} $. Applying subsequently part (a) of Lemma \ref{cut points in HNN}, Corollary \ref{lem-K}, Lemma \ref{lem-U}, and part (a) of Lemma \ref{cut points in HNN} again, we obtain
$$
|g|_A=|g|_C \leq |g|_D^2= |\pi(g)|_{B_2}^2= |\pi (g)|_{B}^2.
$$
The inequality $|\pi (g)|_B\le |g|_A$ is obvious. This concludes the proof of (c).

Finally to prove (d) we note that $G^z \cap G \leq Z(G)$ in $G_0$.  By \cite[Theorem 11.2, Ch. V]{shu-lyn}, the natural map $K\to U/N$ is injective. Hence $G^z \cap G \leq Z(G)=Z(\Gamma)$ in $\Gamma$. Observe that $\Gamma /Z(\Gamma)$ is isomorphic to the HNN-extension $\< \Gamma _0/Z(\Gamma ) \mid g^q=g,\, g\in G/Z(\Gamma )\> $. The later is constricted by part (b) of Lemma \ref{cut points in HNN}.
\end{proof}

\section{Proof of Theorem \ref{main}}

In this section we assume that zero is a natural number. We start with a technical lemma. Let $\alpha : \Bbb N^2 \rightarrow \Bbb N$ to be a bijective recursive function, say the one defined by $$\alpha(j,n)= ((j+n)^2+j+3n)/2.$$ Set $\alpha_1(m)=j$ whenever $m=\alpha(j,n)$ for some $n\in\mathbb N$.

\begin{lem}\label{lem-existence-good-sequence}  Let $(p_j)_{j \in \Bbb N}$ be a recursively enumerable  sequence of natural  numbers.  Let $F$ be the free group on $\{a,b\}$. Then there exists a recursively  enumerable sequence $(R_n|n \in \Bbb N)$ of cyclically reduced words such that the following properties hold.
\begin{enumerate}
\item[(a)]  The function  $\beta : \Bbb N \rightarrow \Bbb N$ defined by $\beta (n)=\|R_n\|$ is recursive.

\item[(b)] The symmetrized set generated by $\{R_n \}_{ n \in \Bbb N}$ satisfies $C'(1/24)$.

\item[(c)]  For $n \in \Bbb N$, we have
$$
\sum_{k=0}^{n-1} p_{\alpha_1(k)} \beta(k)=o(\sqrt{\beta(n)}).
$$

\item[(d)] Given $j \in \Bbb N$, let $d_n^j=\beta(\alpha(j,n))$. Then $\lim_{n\to \infty} d_n^j=+\infty$.
\end{enumerate}
\end{lem}

\proof
Let $(D_n)_{n \in \Bbb N}$ be any infinite recursively enumerable sequence of cyclically reduced words such that the symmetrized set  that it generates satisfies $C'(1/24)$.  We assume also that $(|D_n|)_{n\in \mathbb N}$ is recursively enumerable.
We set  $R_0=D_0$ and $\beta(0)=\| D_0\| $. We will define $R_n$  by induction on $n$ so that
\begin{equation} \label{Rn1}
\left(\sum_{k=0}^{n-1} p_{\alpha_1(k)} \beta(k)\right)/\sqrt{\beta(n)}\leq \frac{1}{n}.
\end{equation}
Suppose that $R_n$ and $\beta(n)$ are already  defined.  We pick the first element $D_m$  which satisfies
$$
\left((n+1)\sum_{k=1}^{n} p_{\alpha_1(k)} \beta(k)\right)^2\leq \|D_m\|
$$
and set $R_{n+1}=D_m$ and $\beta(n+1)=\| D_m\|$. Obviously (\ref{Rn1}) holds.  It follows that the sequence $(R_n)_{n \in \Bbb N}$  satisfy  properties (a)--(c).

Since $\beta (n) \geq n$, we have
$$
\beta(\alpha(j,n)) \geq \alpha(j,n) \geq n,
$$
for any fixed $j$ and thus $\lim^{\omega}d_n^j=+\infty$ as required. \qed

\begin{lem} \label{lem-G-center}Let $(k_n)_{n \in \Bbb N}$ be a   sequence of natural  numbers, where $k_n \geq 2$, and $F$ the free group on $\{a,b\}$. Let $(R_n)_{n \in \Bbb N}$ be a sequence  of cyclically reduced  words such that the symmetrized set that it generates satisfies $C'(1/24)$.  Let
$$
G=\<a,b | R_n^{k_n}=[R_n,a]=[R_n,b]=1, n \in \Bbb N\>.
$$
Then the following properties hold:
\begin{enumerate}
\item[(a)] $Z(G)$ is the subgroup generated by $\{R_n\} _{n \in \Bbb N}$.

\item[(b)]  Let $U$ be a subword of a word $R_n^{k_n}$ of length at most $k_n|R_n|/2$. Then the length of the element represented by the word $U$ in $G$ is at least $|U|/8$. In particular, $|R_n^\e |_{\{a,b\}}\ge \| R\|/8 $ for every $\e \not\equiv 0 (\mod k_n)$.
\end{enumerate}
\end{lem}

\proof
Let $N=\<R_n | n \in \Bbb N\>$. By the given presentation, $N \leq Z(G)$.   We note that  $G/N \cong F/(R)$, where $(R)$ is the normal closure of $(R_n|n \in \Bbb N)$ in $F$. Since  the symmetrized set generated by  $(R_n|n \in \Bbb N)$ satisfies $C'(1/24)$,   $Z(G/N)=1$.  Therefore  $N=Z(G)$. The first part of (b) is exactly the statement of  \cite[Lemma 5.10]{Ol-Os-S}.  The statement `in particular' follows easily. \qed

We are now ready to present the main construction used in the proof of Theorem \ref{main}. Let $\mathcal P= (p_j)_{j \in \Bbb N}$ be a recursively enumerable sequence of prime numbers. Let $(R_n)_{n\in \mathbb N}$ and $\beta :\Bbb N \rightarrow \Bbb N$ be the sequence and the map  given by Lemma \ref{lem-existence-good-sequence}.  Set $$k_n=p_{\alpha_1(n)}$$ and let
$$
G=\<a,b| R_n^{k_n}=[R_n, a]=[R_n,b]=1, n \in \Bbb N\>.
$$

Then $G$ is recursively presented and  the subgroup $N=\<R_0, R_1, \dots \>$ is the center of $G$, by Lemma \ref{lem-G-center}.  Let $A=\{a,b\}$. We embed $G$ into a finitely presented group $\Gamma$ with a generating set $B$  as in Theorem  \ref{thm-embedding}. In particular, we have
\begin{equation}\label{distortion}
\sqrt{|g|_A}\leq |g|_B \leq |g|_A
\end{equation}
for every $g\in G$.

The center $N$ inherits a metric from $\Gamma$ and we consider asymptotic cones of $N$ with respect to that metric.  Let $\omega$ be a non-principal ultrafilter  and $d=(d_n)$  an unbounded non-decreasing  scaling sequence of positive real numbers.

\begin{lem} \label{lem1} Let $(g_n)^\omega \in$ \emph{Con}$^{\omega}(N, d)$ and let  $g_n =R_0^{\epsilon_1} \cdots R_{i_n}^{\epsilon_{i_n}}$, where $0 \leq \epsilon_l\leq k_l-1$ for $0 \leq l \leq i_n$ and $\epsilon_{i_n} \neq 0$.  Then $\sum_{l=0}^{i_n-1}k_l|R_l|_B=o_{\omega}(|R_{i_n}|_B)=o_{\omega}(|R_{i_n}^{\epsilon_{i_n}}|_B)$.
\end{lem}

 \proof We have
$$
dist_B(g_n, R_{i_n}^{\epsilon_{i_n}}) \leq \sum_{l=0}^{i_n-1}|R_l^{\epsilon_l}|_B \leq \sum_{l=0}^{i_n-1}k_l|R_l|_B\leq \sum_{l=0}^{i_n-1}k_l \beta(l).
$$
By Lemma \ref{lem-G-center}(b) and (\ref{distortion}), we have
$$
\sqrt{\beta(i_n)}/8 \leq |R_{i_n}|_B \leq \beta(i_n).
$$
Therefore,
$$
(\sum_{l=0}^{i_n-1}k_l|R_k|_B)/|R_{i_n}|_B  \leq 8(\sum_{l=0}^{i_n-1}k_l \beta(l))/\sqrt{\beta(i_n)},
$$
and the first equality follows  by Lemma \ref{lem-existence-good-sequence}(c).  Combining (\ref{distortion}) and  Lemma \ref{lem-G-center}(b), we get
\begin{equation}\label{former 2}
|R_{i_n}^{\epsilon_{i_n}}|_B \geq \sqrt{|R_{i_n}^{\epsilon_{i_n}}|_A}\geq \sqrt{\beta(i_n)}/8,
\end{equation}
and the second equality follows also by Lemma \ref{lem-existence-good-sequence}(c).\qed

 \begin{lem} \label{lem2} Let $(g_n)^\omega \in$ \emph{Con}$^{\omega}(N, d)$ as in Lemma \ref{lem1}. Then $|R_{i_n}^{\epsilon_{i_n}}|_B=O_\omega(d_n)$ and  $(g_n)^\omega=(R_{i_n}^{\epsilon_{i_n}})^\omega$.  In particular, if $(g_n)^\omega\neq (1)^\omega$ then $|R_{i_n}^{\epsilon_{i_n}}|_B=\Theta_\omega(d_n)$.
 \end{lem}

 \proof We have
$$
\frac{|g_n|_B}{d_n}\geq \frac{|R_{i_n}^{\epsilon_{i_n}}|_B}{d_n}(1-o_\omega(|R_{i_n}^{\epsilon_{i_n}}|_B)),
$$
and thus $\lim^\omega (|R_{i_n}^{\epsilon_n}|_B/d_n)<\infty$.  By Lemma \ref{lem1},  $d(g_n, R_{i_n}^{\epsilon_n})=o_\omega(|R_{i_n}^{\epsilon_{i_n}}|_B)$ and since $\lim^\omega (|R_{i_n}^{\epsilon_n}|_B/d_n)<\infty$, we conclude that $d(g_n, R_{i_n}^{\epsilon_n})=o_\omega(d_n)$. Hence $(g_n)^\omega=(R_{i_n}^{\epsilon_{i_n}})^\omega$. \qed

\begin{lem}  \label{lem3} For every $j \in \Bbb N$, there exists $r^j=(r^j_n)_{n \in \Bbb N}$ such that for any non-principal ultrafilter $\omega$, \emph{Con}$^{\omega}(N, r^j)$ consists of  exactly $p_j$ points.
\end{lem}

\proof   Let  $r^j=(r^j_n)$, where
$$
r_n^j=|R_{\alpha(j,n)}|_B.
$$

Combining (\ref{distortion}) and  Lemma \ref{lem-G-center}(b), we get
\begin{equation}\label{former 3}
\sqrt{d^j_n}/8\leq r_n^j\leq d^j_n,
\end{equation}
and in particular $\lim^\omega r^j_n=+\infty$ by Lemma \ref{lem-existence-good-sequence}(d).  Let $(g_n)^\omega \in$ Con$^{\omega}(N, d)$, with $(g_n)^\omega\neq (1)^\omega$,  and write   $g_n =R_0^{\epsilon_1} \cdots R_{i_n}^{\epsilon_{i_n}}$, $0 \leq \epsilon_l\leq k_l-1$ for $0 \leq l \leq i_n$ and $\epsilon_{i_n} \neq 0$ as in Lemma \ref{lem1}.

By (\ref{distortion}), (\ref{former 2}), and (\ref{former 3}) we have
$$
\sqrt{ \beta(i_n)}/8 \leq |R_{i_n}^{\epsilon_{i_n}}|_B\leq |R_{i_n}^{\epsilon_{i_n}}|_A,
$$
and
$$
\sqrt{d^j_n}/8\leq r_n^j\leq d^j_n.
$$
Thus
\begin{equation}\label{C1}
\lim{^\omega} (\sqrt{\beta(i_n)}/d^j_n)/8\leq \lim{^\omega} (|R_{i_n}^{\epsilon_{i_n}}|_B/r^j_n)<+\infty ,
\end{equation}

Let us show that
\begin{equation}\label{C2} \omega(\{n | i_n=\alpha(j,n)\})=1.
\end{equation}
Indeed otherwise one of the following two equalities holds
$$
\omega(\{n | i_n<\alpha(j,n)\})=1
$$
$$
\omega(\{n | i_n>\alpha(j,n)\})=1.
$$
Suppose that the first equality holds. Then $\beta(i_n) < d^j_n=\beta (\alpha(j,n))$. By properties of the map $\beta$, we get $\lim^{\omega} k_{i_n} \beta (i_n) /\sqrt{d_n^j}=0$. However, we have
$$
|R_{i_n}^{\epsilon_{i_n}}|_B/r^j_n \leq k_{i_n}|R_{i_n}|_B/\sqrt{d_n^j}\leq k_{i_n} \beta (i_n) /\sqrt{d_n^j},
$$
which shows that $\lim^\omega |R_{i_n}^{\epsilon_{i_n}}|_B/r^j_n=0$, contradicting Lemma \ref{lem2}.
Suppose that the second equality holds. Then $\beta(i_n) > d^j_n=\beta (\alpha(j,n))$ and as above $\lim^\omega (\sqrt{\beta(i_n)}/d^j_n)=+\infty$,  contradicting (\ref{C1}).

Thus (\ref{C2}) holds  and hence there exists $0 \leq s \leq p_j$ such that $(g_n)^\omega=(R_{\alpha(j,n)}^s)^\omega$. Suppose that for $s\neq t$,  $(R_{\alpha(j,n)}^s)^\omega=(R_{\alpha(j,n)}^t)^\omega$.  Then $\lim^\omega |R_{\alpha(j,n)}^{s-t}|_T/r^j_n=0$. Since the subgroup generated by $R_{\alpha(j,n)}$ is cyclic of prime order $p_j$, $R_{\alpha(j,n)}^{s-l}$ is also a generator, and thus
$$
R_{\alpha(j,n)}=(R_{\alpha(j,n)}^{s-l})^{m_n},$$
for some $1 \leq m_n \leq p_j-1$.  Therefore  $$p_j|R_{\alpha(j,n)}^{s-t}|_B \geq m_n|R_{\alpha(j,n)}^{s-t}|_B \geq |R_{\alpha(j,n)}|_B,$$
and thus $\lim^\omega |R_{\alpha(j,n)}|_B/r^j_n=0$, which is a contradiction. This ends the proof of the lemma.   \qed

 \begin{lem} \label{lem4} There exists $d=(d_n)$ such that for any for any non-principal ultrafilter $\omega$, \emph{Con}$^{\omega}(N, d)$ consists of  exactly one  point.
 \end{lem}

 \proof Let $d=(d_n)$, where
 $$
 d_n=nk_n\beta(n).
 $$

 We claim that Con$^{\omega}(N, d)$ consists of  exactly one  point.    Let $(g_n)^\omega \in$ Con$^{\omega}(N, d)$ and write   $g_n =R_0^{\epsilon_1} \cdots R_{i_n}^{\epsilon_{i_n}}$, $0 \leq \epsilon_l\leq k_l-1$ for $0 \leq l \leq i_n$ and $\epsilon_{i_n} \neq 0$ as in Lemma \ref{lem1}.  One of the following equalities holds:
 $$
\omega(\{n | i_n<n\})=1, \; \; \omega(\{n | i_n>n\})=1, \;\; \omega(\{n | i_n=n\})=1.
$$
We will show that in each case $(g_n)^\omega=(1)^\omega$.

 Suppose that the first equality holds.  Then $\beta(i_n)<\beta (n)$ and
$$
 |R_{i_n}^{\epsilon_n}|_B/d_n \leq k_{i_n}\beta(i_n)/d_n, \; \; \lim{^\omega} (k_{i_n}\beta(i_n)/nk_n\beta(n)) =0,
$$
and thus $(g_n)^\omega=(1)^\omega$ as required.
Further suppose that the second equality holds.   Then $\beta(i_n)>\beta (n)$ and
$$
 |R_{i_n}^{\epsilon_n}|_B/d_n \geq \sqrt{\beta(i_n)}/d_n, \; \; \lim{^\omega} (\sqrt{\beta(i_n)}/nk_n\beta(n)) =\infty,
 $$
and thus $(g_n) \not \in$ Con$^\omega(N, d)$; a contradiction. Finally if the third equality holds, then
 $$
 |R_{i_n}^{\epsilon_n}|_B/d_n\leq k_n\beta(n)/nk_n\beta(n)=1/n,
 $$
 and thus $(g_n)^\omega=(1)^\omega$ as required. \qed

Lemmas \ref{lem3} and \ref{lem4} will be combined with the following result.

\begin{lem}[Olshanskii-Osin-Sapir {\cite[Theorem 5.8]{Ol-Os-S}}] \label{thm-OOS} Let $N$ be a central subgroup of a finitely generated group $G$ endowed with the metric induced by the word metric on $G$. Suppose that \emph{Con}$^{\omega}(N,d)$ consists of $m<\infty$ points for some non-principal ultrafilter $\omega$ and some scaling sequence $d=(d_n)$. Then $c(\CG)=mc(\Con(G/N, d))$.
\end{lem}

We are  now ready to prove the main result of this section.

\begin{proof}[Proof of Theorem \ref{main}] Let $(p_j)_{j \in \Bbb N}$ be a recursively enumerable sequence of primes such that $p_i \neq p_j$ for $i \neq j$, and let $r^j$ be the sequence given by Lemma \ref{lem3}. Since $N=Z(G)$  we have $N=Z(\Gamma)$ by Theorem  \ref{thm-embedding} (b). Hence,  by Lemma \ref{thm-OOS}, $c($Con$^\omega(\Gamma,r^j))=p_jc($Con$^\omega(\Gamma/Z(\Gamma),r^j))=p_j$ as $\Gamma/Z(\Gamma ) $ is constricted by Theorem \ref{thm-embedding} (d). On the other hand, if $d$ is the sequence given by Lemma \ref{lem4}, then we similarly obtain $c($Con$^\omega(\Gamma,d))=1$. By choosing the sequence  $(p_j)_{j \in \Bbb N}$ to be the one of all the primes, we obtain Theorem \ref{main}.
\end{proof}

\begin{rem} We note that for different sequences $p=(p_j)_{j \in \Bbb N}$, $q=(q_j)_{j \in \Bbb N}$, the corresponding groups $\Gamma_p$ and $\Gamma_q$ satisfy $Z(\Gamma_p) \not \cong Z(\Gamma_q)$ and thus $\Gamma_p$ and $\Gamma_q$ are not isomorphic. Hence there exist countably many finitely presented groups satisfying conclusions of both corollaries of Theorem \ref{main}.
\end{rem}


\begin{thebibliography}{99}
\bibitem{Van-Wilkie}  L. van den Dries, A. J. Wilkie, \textit{On Gromov's theorem concerning groups of polynomial
growth and elementary logic, } J. of Algebra \textbf{89} (1984), 349-374.

\bibitem{Dru}
C. Drutu, \textit{Quasi-isometry invariants and asymptotic cones}, Int. J. Alg. Comp. \textbf{12} (2002), 99-135.

\bibitem{DMS}
C. Drutu, S. Mozes, M. Sapir,   \textit{Divergence in lattices in semisimple Lie groups and graphs of groups}.   Trans. Amer. Math. Soc.  \textbf{362} (2010), no. 5, 2451-2505.

\bibitem{DS}
C. Drutu, M. Sapir, \textit{Tree-graded spaces and asymptotic cones of groups. With an appendix by D. Osin and M. Sapir.} Topology  {\bf 44} (2005), no. 5, 959-1058.

\bibitem{Gro-asy} M. Gromov, \textit{ Groups of polynomial growth and expanding maps},  Publ. Math. IHES 53
(1981), 53-73.

\bibitem{Gro}
M. Gromov, Hyperbolic groups, \textit{Essays in Group Theory,} MSRI Series, Vol.8, (S.M.
Gersten, ed.), Springer, 1987, 75--263.

\bibitem{Gro-que} M.Gromov. Asymptotic invariants of inﬁnite groups, in: Geometric Group Theory. Vol. 2
(G.A.Niblo and M.A.Roller, eds.), London Math. Soc. Lecture Notes Ser., 182 (1993), 1-295.


\bibitem{K-S-T-T} L. Kramer, S. Shelah, K. Tent, S. Thomas,\textit{ Asymptotic cones of finitely presented groups,}  Adv. Math. \textbf{193} (2005) 142-173.

\bibitem{shu-lyn} R. C. Lyndon and P. E. Schupp,  Combinatorial group theory. Springer-
Verlag, Berlin, 1977. Ergebnisse der Mathematik und ihrer Grenzgebiete,
Band 89.

\bibitem{Ol} A.Yu. Olshanskii, \textit{On the distortion of subgroups of finitely presented groups}.
Mat. Sb. 188 (1997), no. 11, 51--98; translation in Sb. Math. 188 (1997), no. 11, 1617-1664.


\bibitem{Ol-Os-S} A. Yu. Olshanskii, D.V. Osin, M.V. Sapir,  \textit{Lacunary hyperbolic groups}. With an appendix by M. Kapovich and B. Kleiner. Geom. \& Topol. 13 (2009), no. 4, 2051--2140.


\bibitem{Ol-sap} A. Olshanskii, M.Sapir, \textit{A finitely presented group with two non-homeomorphic asymptotic cones},  Internat. J. Algebra Comput. \textbf{17} (2007), no. 2, 421-426.


\bibitem{Ould} A. Ould Houcine,  \textit{Embeddings in finitely presented groups which preserve the center}. Journal of Algebra, 307 (2007) 1-23.

\bibitem{Thom-Veli} S. Thomas, B. Velickovic, \textit{Asymptotic cones of ﬁnitely generated groups}, Bull. London
Math. Soc. \textbf{32} (2000), no. 2, 203-208.


\end{thebibliography}
\end{document}